\documentclass[12pt,a4paper]{article}
\usepackage[cp1251]{inputenc}
\usepackage[english, russian, ukrainian]{babel}
\usepackage{amsmath,amssymb,amsthm}
\makeatletter
\@addtoreset{equation}{section}
\makeatother
\makeatletter
\@addtoreset{section}{part}
\makeatother
\makeatletter
\@addtoreset{thm}{section}
\makeatother
\makeatletter
\@addtoreset{lem}{section}
\makeatother
\makeatletter
\@addtoreset{defn}{section}
\makeatother

\theoremstyle{plain}
\newtheorem{lem}{Lemma}[section]
\newtheorem{thm}{Theorem}[section]

\theoremstyle{definition}
\newtheorem{defn}{Definition}[section]

\theoremstyle{remark}

\newcommand{\cK}{{\mathcal K}}

\newcommand{\cT}{{\mathcal T}}

\newcommand{\mbR}{{\mathbb R}}

\newcommand{\mbC}{{\mathbb C}}

\newcommand{\vk}{\varkappa}
\newcommand{\vf}{\varphi}
\newcommand{\ve}{\varepsilon}
\newcommand{\wt}{\widetilde}

\renewcommand{\ker}{\mathop{\rm ker}}
\renewcommand{\Re}{\mathop{\rm Re}}
\newcommand{\ov}{\overline}
\newcommand{\1}{1\!\!\,{\rm I}}
\begin{document}
\large
\begin{center}
{\Large\bf
Moments estimates for local times of one class of Gaussian processes}
\end{center}

{\bf author:} Olga Izyumtseva

{\bf address:} Institute of Mathematics,
National Academy of Sciences of Ukraine, Kiev-4, 01601,
3, Tereschenkivska st, Ukraine

{\bf email:} olaizymtseva@yahoo.com

{\bf 2010 Mathematics Subject Classification:} Primary 60G15, 60J55; Secondary 60H40

{\bf keywords:} Local time, Gaussian integrators, Brownian bridge, white noise, Gram determinant, generalized functionals, Fourier-Wiener transform

{\bf Abstract}. In present paper we prove an existence and give a moments estimate for the local time of Gaussian integrators. Every Gaussian integrator is associated with a continuous linear operator in the space of square integrable functions via white noise representation. Hence, all properties of such process are completely characterized by properties of the corresponding operator. We describe the sufficient conditions on continuous linear non-invertible operator which allow the local time of the integrator to exists at any real point. Moments estimate for local time is obtained. A continuous dependence of local time of Gaussian integrators on generating them operators is established. The received statement improves our result presented in \cite{1}.
\section{Introduction}

In present paper we study properties of local times \cite{16} for Gaussian integrators. This class of Gaussian processes firstly was introduced by A.A.Dorogovtsev in the work \cite{2}. The original definition is the following.
\begin{defn}\cite{2}
\label{defn1.1}
A centered Gaussian process $x(t),\ t\in[0;1],\  x(0)=0$ is said to be an integrator if there exists the constant $c>0$ such that for an arbitrary partition $0=t_0<t_1<\ldots<t_n=1$ and real numbers $a_0, \ldots,a_{n-1}$ the following relation holds
$$
E\Big(\sum^{n-1}_{k=0}a_k(x(t_{k+1})-x(t_k))\Big)^2\leq c\sum^{n-1}_{k=0}a_k\Delta t_k.
$$
\end{defn}
This inequality allows to define a stochastic integral for any square integrable function with respect to integrator. Stochastic calculus for\newline integrators including the Ito formula and stochastic anticipating equations were introduced in \cite{2,3}. The following statement describes the structure of integrators using the notion of white noise in the space of square integrable functions.
\begin{lem}\cite{2}
\label{lem1.2}
 The centered Gaussian process $x(t),\ t\in[0; 1]$ is the integrator iff there exist a white noise $\xi$ [5-7] in the space $L_2([0; 1])$ and a continuous linear operator in the same space such that
 $$
 x(t)=(A\1_{[0;1]}, \xi), \ t\in[0; 1].
 $$
  \end{lem}

It follows from Lemma \ref{lem1.2} that all properties of the integrator $x$ can be characterized in terms of the operator $A.$ Note that in this case the covariance function of the process $x$ has the form
$$
Ex(t)x(s)=(A1_{[0;t]},A1_{[0;s]}).
$$
The main object of investigation in the present paper is the local time for Gaussian integrators. Let us start from the general definition of local time for a 1-dimensional random process $x(t),\ t\in[0; 1].$ Denote by
$$
f_\ve(z)=\frac{1}{\sqrt{2\pi\ve}}e^{-\frac{z^2}{2\ve}}.
$$
\begin{defn}
\label{defn1.3}
For $t\in[0; 1],\ u\in\mbR$ and $p\geq 2$
$$
l^{x}(u,t)=\int^t_0\delta_u(x(s))ds=L_p\mbox{-}\lim_{\ve\to0}\int^t_0f_\ve(x(s)-u)ds
$$
is said to be the local time of the process $x$ at the point $u$ up to the time $t$ whenever the limit exists.
\end{defn}
Denote by $l^{x}=l^{x}(0,1),\ l^{x}(u)=l^{x}(u,1).$
Let $G(e_1, \ldots, e_n)$ be the Gram determinant constructed from elements $e_1,\ldots, e_n$ of the space $L_2([0; 1]).$ Denote by $g(t)=A\1_{[0;t]}.$ In Section 2 we show that an existence of local time follows from the convergence of the following integral
\begin{equation}
\label{eq1.1}
\int_{\Delta_p}
\frac{d\vec{s}}
{\sqrt{G(g(s_1), \ldots, g(s_p)}},
\end{equation}
where $\Delta_p=\{0\leq s_1\leq\ldots\leq s_p\},\ p\geq 2.$
Note that in the case $p=2$ one can use the following statement which was proved in \cite{15}. Let $(T,\Xi)$ be a measurable space with the finite measure $\nu.$ Consider a centered Gaussian random field $\zeta$ on $(T,\Xi).$ Denote by $B(s,t)$ the covariance matrix of the vector $(\zeta(s),\zeta(t)).$ Put $I_{\ve}=\int_{T}f_{\ve}(\zeta(t))\nu(dt).$
\begin{thm}\cite{15}
\label{thm3}
Suppose that
\begin{equation}
\label{eq10}
\int_{T}\int_{T}\frac{1}{\sqrt{det B(s,t)}}\nu(ds)\nu(dt)<\infty,
\end{equation}
then the limit of $I_{\ve}$ exists in $L_{2}(\Omega),$ as $\ve\to0.$
\end{thm}
Therefore, our statement generalizes Theorem 1.1 for integrators in the case of $L_p-$ convergence for $p>2.$ It must be mentioned that for a $1$-dimensional Wiener process P. Levy in \cite{17} proved the existence of local time. T. Trotter in \cite{18} using Kolmogorov's continuity criterion proved that the local time of the Wiener process is jointly continuous with respect to time and space variable. For general Gaussian process S. Berman in \cite{19} proposed the notion of local nondeterminism which in some sense means almost independency of increments on small intervals. Under some technical assumption this property leads to existence of local times and their joint continuity. The local time of anisotropic Gaussian random fields studied by Y. Xiao in \cite{20}. Typical examples of such fields are fractional Brownian sheets, operator-scaling Gaussian fields with stationary increments and the solution to the stochastic heat equation. By representing the notion of strong local nondeterminism Y. Xiao proved the existence of local times and their joint continuity for anisotropic Gaussian random fields. The Gaussian integrators which we consider in the present paper have some version of local nondeterminism only if they generated by a continuously invertible operator. Using this fact, for integrator generated by a continuously invertible operator in the work \cite{7} we prove the existence of local time which is jointly continuous in time and space variable. The key idea of the proof is the following statement.
\begin{lem}\cite{7}
\label{lem1.4}
Suppose that $A$ is a continuously invertible operator in the Hilbert space $H.$ Then for all $k\geq1$ there exists a positive constant $c(k)$ which depends on $k$ and $A$ such that for any $e_1, \ldots, e_k\in H$ the following relation holds
$$
G(Ae_1, \ldots, Ae_k)\geq c(k)G(e_1, \ldots, e_k).
$$
\end{lem}

In this paper we will check that $c(k)=\frac{1}{\|A^{-1}\|^{2k}}$ (see Theorem \ref{thm2.3}).
It follows from Lemma \ref{lem1.4} that \eqref{eq1.1} can be estimated by
$$
\int_{\Delta_p}
\frac{d\vec{t}}
{\sqrt{G(\1_{[0;t_1]}, \ldots, \1_{[0; t_p]})}}=
\frac{(2\pi)^{p/2}}{p!}E(l^w)^p,
$$
where $w(t),\ t\in[0; 1]$ is a 1-dimensional Wiener process. In the paper the same integral with the Gram determinant constructed by elements $A1_{[t_i;t_{i+1}]},\ i=\overline{1,p-1}$, where the operator $A$ is non-invertible will be estimated by the moments of local time of Brownian bridge. Let $\ker A$ be the kernel of operator $A$ and $(\ker A)^\perp=\{y\in L_2([0;1]): \forall \ z\in\ker A:\  (y,z)=0\}.$ Denote by $L$ the subspace of all step functions in $\ker A.$ Let $\{f_k,\ k=\ov{1,s}\}$ be an orthonormal basis in $L.$ Suppose that $0\leq s_1<\ldots<s_N\leq1$ are the points of jumps of functions $f_1, \ldots, f_s.$ Note that any other orthonormal basis in $L$ has the same set of the points of jumps. Really, let $\tilde{f_1},\ldots,\tilde{f_s}$ be also the orthonormal basis in $L.$ Denote by $K,\ \tilde{K}$ the sets of points of jumps of functions $\{f_j,\ j=\ov{1,s}\}$ and $\{\tilde{f_j},\ j=\ov{1,s}\}$ correspondingly. Then
\begin{equation}
\label{eq1.3}
f_j=\sum^{s}_{k=1}(f_j,\tilde{f_k})\tilde{f_k},\ j=\ov{1,s}
\end{equation}
and
\begin{equation}
\label{eq1.4}
\tilde{f_j}=\sum^{s}_{k=1}(\tilde{f_j},f_k)f_k,\ j=\ov{1,s}.
\end{equation}
It follows from (1.3), (1.4) that $K\subset\tilde{K}$ and $\tilde{K}\subset K$. Consequently, $K=\tilde{K}.$ Consider independent 1-dimensional  Wiener processes  $w_1,\ldots,w_{N+1}$. Let us define the process
$y(t),\ t\in[0;1]$ via the operator $A$ and $w_1,\ldots,w_{N+1}$ as follows
\begin{equation}
y(t)=
\begin{cases}
w_1(t)-\frac{t}{s_1}w_1(s_1),&t\in[0;s_1]\\
w_2(t-s_1)-\frac{t-s_1}{s_2-s_1}w_2(s_2-s_1),&t\in[s_1;s_2]\\
\cdot\\
\cdot\\
\cdot\\
w_{N+1}(t-s_N)-\frac{t-s_N}{1-s_N}w_{N+1}(1-s_N),&t\in[s_N;1].
\end{cases}
\end{equation}
Denote by $\tilde{A}$ the restriction of operator $A$ on $(\ker A)^\perp$.
In present paper we will show that for the continuous linear operator $A$ which satisfies the conditions

1) $\dim\ker A<+\infty$

2) the operator $\tilde{A}$ is continuously invertible

the following relation holds
$$
\int_{\Delta_p}
\frac{d\vec{t}}
{\sqrt{G(g(t_1), \ldots, g(t_p))}}\leq \frac{1}{p!\ \|\tilde{A}^{-1}\|^{p} }E(l^{y})^p,\ p\geq 1.
$$
The obtained estimate allows to prove the existence of local time and give a moments estimate for the local time of integrator generated by the operator $A$ which satisfies the conditions 1), 2).

This paper is organized as follows. In Section 2 we discuss sufficient conditions on a continuous linear non-invertible operator that allow the local time of corresponding integrator to exists at any real point. Moments estimate for local time is introduced. In Section 3 we expand our result proved in \cite{1} related to a continuous dependence of local time of integrators on generating them  operators.

\section{Existence of local time for Gaussian integrators. Moments estimate for local time}

In present section we prove an existence of local time and give a moments estimate for a local time of integrator generated by a continuous linear operator in the space $L_2([0; 1])$ with a finite-dimensional kernel. Consider the Gaussian integrator
$$
x(t)=(A\1_{[0;t]}, \xi), \ t\in[0;1],
$$
where $A$ is the continuous linear operator in the space $L_2([0;1]).$ As in Section 1 $\tilde{A}$ is the restriction of operator $A$ to $(\ker A)^\perp.$

\begin{thm}
\label{thm2.1}
Suppose that the continuous linear operator $A$ in the space $L_2([0; 1])$ satisfies conditions

1) $\dim\ker A<+\infty;$

2) the operator  $\tilde{A}$ is continuously invertible.

Then for any $u\in\mbR$ and $t\in(0;1]$ the Gaussian integrator $x$ generated by the operator $A$ has the local time $l(u, t).$ Moreover, for any $p\geq 1$ the following relation holds
$$
E(l^{x})^p\leq \frac{1}{p!\ \|\tilde{A}^{-1}\|^{p} }E(l^{y})^p,
$$
where the process $y$ defined in (1.5).
\end{thm}
\renewcommand{\proofname}{Proof}
\begin{proof}
Without loss of generality we can assume that $u=0,\ t=1.$ Consider
$$
f_\ve(z)=\frac{1}{\sqrt{2\pi\ve}}e^{-\frac{z^2}{2\ve}},\ \ve>0,\ z\in\mbR
$$
and check that for any $p\in \mathbb{N}$ there exists
$$
L_p\mbox{-}\lim_{\ve\to0}
\Big[\int^1_0f_\ve(x(s))ds\Big]=:l^x.
$$
By $V_\ve$ denote $\int^1_0f_\ve(x(s))ds.$ Let us check that for any $p\geq1$
\begin{equation}
\label{eq2.1}
E(V_{\ve_1}-V_{\ve_2})^p\to0, \ \ve_1, \ve_2\to0.
\end{equation}
Note that
$$
E(V_{\ve_1}-V_{\ve_2})^p=\sum^p_{l=0}C^l_p(-1)^{p-l}E(V_{\ve_1})^l(V_{\ve_2})^{p-l}.
$$
Hence, to check the relation \eqref{eq2.1} it suffices to prove that there exists the finite limit $E(V_{\ve_1})^l(V_{\ve_2})^{p-l}$ as $\ve_1, \ve_2\to0.$ Using the relation
$$
f_{\ve_2}(z)=\sqrt{\frac{\ve_1}{\ve_2}}f_{\ve_1}\Big(\sqrt{\frac{\ve_1}{\ve_2}}z\Big)
$$
one can get
\begin{equation}
\label{eq2.2}
E(V_{\ve_1})^l(V_{\ve_2})^{p-l}=p!\Big(\frac{\ve_1}{\ve_2}\Big)^{\frac{p-l}{2}}
\int_{\Delta_p}E\wt{f}_{\ve_1}(\wt{x}_\ve(s))d\vec{s},
\end{equation}
where
$$
\wt{f}_\ve(z)=\frac{1}{(2\pi\ve)^{p/2}}e^{-\frac{\|z\|^2}{2\ve}},\ \ve>0,\  z\in\mbR^p,
$$
$$
\wt{x}_{\ve}(s)=\begin{pmatrix}
x(s_1)\\
\vdots\\
x(s_l)\\
\sqrt{\frac{\ve_1}{\ve_2}}x(s_{l+1})\\
\vdots\\
\sqrt{\frac{\ve_1}{\ve_2}}x(s_{p})
\end{pmatrix}.
$$
Calculating the expectation of the integrand in \eqref{eq2.2} one can obtain that \eqref{eq2.2} equals
\begin{equation}
\label{eq2.3}
p!\Big(\frac{\ve_1}{\ve_2}\Big)^{\frac{p-l}{2}}
\frac{1}{(2\pi)^{p/2}}
\int_{\Delta_p}
\frac{1}{\sqrt{\det(B_{\ve}(s)+I\ve_1)}}d\vec{s},
\end{equation}
where $B_{\ve}(s)$ is the covariance matrix of the vector $\wt{x}_{\ve}.$ Let $G_k(s)$ be a sum of all possible Gram determinants constructed from $k$ elements of the set
$$
\Big\{g(s_1), \ldots, g(s_l),\sqrt{\frac{\ve_1}{\ve_2}}g(s_{l+1}), \ldots, \sqrt{\frac{\ve_1}{\ve_2}}g(s_p)\Big\}.
$$
 Note that
$$
G_p(s)=G\Big(g(s_1), \ldots, g(s_l),\sqrt{\frac{\ve_1}{\ve_2}}g(s_{l+1}), \ldots, \sqrt{\frac{\ve_1}{\ve_2}}g(s_p)\Big).
$$
Using a linearity property of a determinant one can conclude that
\begin{equation}
\label{eq2.4}
\det(B_{\ve}(s)+I\ve_1)=\sum^p_{k=0}\ve^{p-k}_1G_k(s).
\end{equation}
Here $G_0(\vec{s})=1.$ It follows from relation \eqref{eq2.4} that \eqref{eq2.3} equals
$$
p!\Big(\frac{\ve_1}{\ve_2}\Big)^{\frac{p-l}{2}}
\frac{1}{(2\pi)^{p/2}}
\int_{\Delta_p}
\frac{d\vec{s}}{\sqrt{\sum^p_{k=0}\ve^{p-k}_1G_k(s)}}=
$$
$$
=p!
\frac{1}{(2\pi)^{p/2}}
\int_{\Delta_p}
\frac{d\vec{s}}{\sqrt{G(g(s_1), \ldots,g(s_p))+\Big(\frac{\ve_2}{\ve_1}\Big)^{p-l}\sum^p_{k=0}\ve^{p-k}_1G_k(s)}}.
$$
Note that for any $\ve_1, \ve_2>0$
$$
\frac{1}{\sqrt{G(g(s_1), \ldots,g(s_p))+\Big(\frac{\ve_2}{\ve_1}\Big)^{p-l}\sum^p_{k=0}\ve^{p-k}_1G_k(s)}}\leq
$$
$$
\leq \frac{1}{\sqrt{G(g(s_1), \ldots,g(s_p))}}
$$
and
$$
\frac{d\vec{s}}{\sqrt{G(g(s_1), \ldots,g(s_p))+\Big(\frac{\ve_2}{\ve_1}\Big)^{p-l}\sum^p_{k=0}\ve^{p-k}_1G_k(s)}}\to
$$
$$
\to
\frac{1}{\sqrt{G(g(s_1), \ldots,g(s_p))}}
$$
as $\ve_1, \ve_2\to0.$ To apply the Lebesgue dominated convergence theorem we  need the following statement which is proved in Appendix A (see Theorem A.2).

\begin{thm}
\label{thm2.2}
Suppose that continuous linear operator $A$ in $L_2([0;1])$ satisfies conditions

1) $\dim\ker A<+\infty;$

2) the operator $\tilde{A}$ is continuously invertible.

Then for any $p\geq1$ the following integral
$$
\int_{\Delta_p}
\frac{d\vec{t}}{\sqrt{G(g(t_1), \ldots,  g(t_p))}}
$$
converges.
\end{thm}

Applying the Lebesgue dominated convergence theorem and Theorem \ref{thm2.2} to \eqref{eq2.3} one can finish the proof of the existence of
$$
L_p\mbox{-}\lim_{\ve\to0}\int^1_0f_\ve(x(s))ds.
$$
Using the same arguments as in the proof of Theorem \ref{thm2.3} (see Appendix A, Theorem A.2) one can obtain the moments estimate for the local time $l^x$. Really,
\begin{equation}
\label{eq2.5'}
E(l^x)^p=
p!
\frac{1}{(2\pi)^{p/2}}
\int_{\Delta_p}
\frac{d\vec{s}}{\sqrt{G(g(s_1),  \ldots, g(s_p))}}.
\end{equation}

Denote by $P$ a projection onto $\ker A.$ Then \eqref{eq2.5'} equals
\begin{equation}
\label{eq2.6}
p!
\frac{1}{(2\pi)^{p/2}}
\int_{\Delta_p}
\frac{d\vec{s}}{\sqrt{G(A(I-P)\1_{[0;s_1]}, \ldots, A(I-P)\1_{[0;s_p]})}}.
\end{equation}
Let us applying the next theorem which is proved in Appendix A (see Theorem A.1).
\begin{thm}
\label{thm2.3}
Suppose that $A$ is continuously invertible operator in the Hilbert space $H.$ Then for any elements $e_1, \ldots, e_n$ of the space $H$ the following relation holds
$$
G(Ae_1, \ldots, Ae_n)\geq\frac{1}{\|A^{-1}\|^{2n}}G(e_1, \ldots, e_n).
$$
\end{thm}
It follows from Theorem \ref{thm2.3} that \eqref{eq2.6} less or equal to
\begin{equation}
\label{eq2.7}
p!
\frac{1}{(2\pi)^{p/2}}\frac{1}{\|\tilde{A}^{-1}\|^{p}}
\int_{\Delta_p}
\frac{d\vec{s}}{\sqrt{G((I-P)\1_{[0;s_1]}, \ldots, (I-P)\1_{[0;s_p]})}}.
\end{equation}
Using the same arguments as in the proof of Theorem \ref{thm2.2} (see Theorem A.2 in Appendix A) one can check that \eqref{eq2.7} less or equal to
$$
\frac{1}{(2\pi)^{p/2}}\frac{1}{\|\tilde{A}^{-1}\|^{p}}E(l^{y})^p.
$$
which proves the theorem.
\end{proof}
\section{On continuous dependence of local times of integrators on generating them operators}
\label{section3}
Suppose that continuous linear operators $A_n,\ A$ in space $L_2([0; 1])$ generate integrators $x_n(t),\  x(t),\ t\in[0; 1].$ Let $l^{x_n}(u),\ l^x(u)$ be the local times of integrators $x_n,\ x$ correspondingly at the point $u\in\mbR.$ In our work \cite{1} we showed that if the sequence of continuously invertible operators $A_n$ converges strongly to continuously invertible operator $A,$ then the sequence of local times $l^{x_n}$ converges in mean square to the local time $l^{x}.$ The next statement improves the previous result. We will check that under the same conditions convergence of local times holds in $L_{2m}$ for any $m\geq1.$

\begin{thm}
\label{thm3.1}
Suppose that $A_n,\  A$ are continuously invertible operators in $L_2([0; 1])$ such that

1) for any $z\in L_2([0; 1])$
$$
\|A_nz-Az\|\to0, \ n\to\infty
$$
2) $ \sup_{n\geq1}\|A_n^{-1}\|<+\infty.$

Then for any $m\geq1$
$$
E\int_{\mbR}(l^{x_n}(u)-l^x(u))^{2m}du\to0, \ n\to\infty.
$$
\end{thm}
\begin{proof}
Note that
$$
E\int_{\mbR}(l^{x_n}(u)-l^x(u))^{2m}du=\sum^{2m}_{k=0}C^k_{2m}(-1)^{2m-k}
E\int_{\mbR}l^{x_n}(u)^kl^x(u)^{2m-k}du.
$$
Consider two cases. Suppose that $0<k<2m.$ Let us prove that
$$
E\int_{\mbR}l^{x_n}(u)^kl^{x}(u)^{2m-k}du\to E\int_{\mbR}l^{x}(u)^{2m}du
$$
as $n\to\infty.$
One can check that the following relation holds
\begin{equation}
\label{eq3.1}
l^{x}(u)=l^{x}(u)\1_{[\min_{[0;1]}x;\max_{[0; 1]}x]}(u).
\end{equation}
Really, note that
$$
\frac{1}{2\ve}\int^1_0\1_{[u-\ve; u+\ve]}(x(t))dt=
$$
\begin{equation}
\label{eq3.2}
=\frac{1}{2\ve}\int^1_0\1_{[u-\ve; u+\ve]}(x(t))dt
\1_{[\min_{[0;1]}x-\ve;\max_{[0; 1]}x+\ve]}(u).
\end{equation}
Passing to the limit in mean square at the right hand side and the left hand side of equality \eqref{eq3.2} on can obtain \eqref{eq3.1}. Formula \eqref{eq3.1} allows to discuss properties of random variable $l^x(u)$ defined on measurable space
$$
\wt{\Omega}=\{(\omega,u): \omega\in\Omega,\  u\in[\min_{[0;1]}x(\omega,t),\max_{[0;1]}x(\omega,t)]\},
$$
$$
\wt{\mathcal{F}}=\{\Delta\cap\wt{\Omega}: \Delta\in \mathcal{F}\otimes\mathcal{B}(\mbR)\},
$$
$$
\wt{P}(\Delta)=E\int^{\max_{[0;1]}x}_{\min_{[0;1]}x}\1_\Delta(\omega,u)du.
$$
One can check that $\wt{P}$ non-probability measure but finite. Really,
$$
P(\wt{\Omega})=E\int^{\max_{[0;1]}x}_{\min_{[0;1]}x}du=E(\max_{[0;1]}x-\min_{[0;1]}x)=2E\max_{[0;1]}x(t).
$$
Let us check that for any $p\geq0$
$$
E(\max_{[0;1]}x(t))^{2p}<+\infty.
$$
That can be achieved using estimate
\begin{equation}
\label{eq3.3}
\frac{1}{\|A\|^2}
E(x(t)-x(s))^2\leq E(w(t)-w(s))^2
\end{equation}
and the following Sudakov--Fernique form of the comparison principle.
\begin{thm}\cite{8}
\label{thm3.2}
Let $\vk=(\vk_1, \ldots, \vk_n)$ and $\eta=(\eta_1,\ldots, \eta_n)$ be centered Gaussian vectors such that the inequality
$$
E(\vk_i-\vk_j)^2\geq E(\eta_i-\eta_j)^2
$$
holds for any $i\ne j.$ Let $g: \mbR_+\to\mbR$ be a nondecreasing convex function. Then
$$
E\sup_i\vk_i\geq E\sup_i\eta_i
$$
and
$$
Eg(\sup_{i,j}(\vk_i-\vk_j))\geq Eg(\sup_{i,j}(\eta_i-\eta_j)).
$$
\end{thm}
Applying \eqref{eq3.3} and Theorem 3.2 one can conclude that
$$
E(\max_{[0;1]}x(t))^{2p}\leq E(\max_{s,t\in[0;1]}x(t)-x(s))^{2p}\leq
$$
$$
\leq\|A\|^{2p}E(\max_{s,t\in[0;1]}w(t)-w(s))^{2p}<+\infty.
$$
The representation for the local time \eqref{eq3.1} implies the relation
$$
E\int_{\mbR}l^{x_n}(u)^kl^x(u)^{2m-k}du=
E\int^{\max_{[0;1]}x}_{\min_{[0;1]}x}
l^{x_n}(u)^kl^x(u)^{2m-k}du.
$$
In \cite{1} we proved that under conditions of the theorem
$$
E\int_{\mbR}(l^{x_n}(u)-l^x(u))^{2}du\to0,\ n\to\infty.
$$
It implies that $l^{x_n}\overset{\wt{P}}{\rightarrow}l^x,\ n\to\infty.$
Therefore, to prove the theorem it suffices to check that the sequence $\{(l^{x_n})^k(l^x)^{2m-k}\}_{n\geq1}$ is uniformly integrable, i.e. for $p\geq1$
$$
\sup_{n\geq1}
E\int^{\max_{[0;1]}x}_{\min_{[0;1]}x}
(l^{x_n}(u)^kl^x(u)^{2m-k})^{2p}du<+\infty.
$$
Applying H\"older estimate one can obtain the following relation
$$
\Big|
E\int^{\max_{[0;1]}x}_{\min_{[0;1]}x}
l^{x_n}(u)^{2kp}l^x(u)^{(2m-k)2p}du
\Big|\leq
$$
$$
\leq
\sqrt{E\int^{\max_{[0;1]}x}_{\min_{[0;1]}x}
l^{x_n}(u)^{4kp}du}
\sqrt{E\int^{\max_{[0;1]}x}_{\min_{[0;1]}x}
l^x(u)^{4p(2m-k)}du}.
$$
Denote by $4kp=q_1,\ 4p(2m-k)=q_2.$
Then
$$
E\int^{\max_{[0;1]}x}_{\min_{[0;1]}x}
l^{x_n}(u)^{q_1}du\leq E\int_{\mbR}l^{x_n}(u)^{q_1}du.
$$
Further we need the following statement which is proved in Appendix B (see Theorem B.1).
\begin{thm}
\label{thm3.2}
Let $f_1, \ldots, f_n$ be the linearly independent elements in $L_2([0; 1]).$ Then
$$
\int_{\mbR}\prod^n_{k=1}\delta_0((f_k,\xi)-u)du=\prod^{n-1}_{k=1}(f_{k+1}-f_k, \xi),\eqno (3.4)
$$
where the equality (3.4) is understood as equality of generalized functionals from white noise \cite{12}.
\end{thm}
It follows from Theorem 3.3 and Theorem A.1 that
$$
E\int_{\mbR}l^{x_n}(u)^{q_1}du=
$$
$$
=q_1!E\int_{\Delta_{q_1}}\prod^{q_1-1}_{i=1}\delta_0(x_n(t_{i+1})-x_n(t_i))d\vec{t}=
$$
$$
=
q_1!
\frac{1}{(2\pi)^{\frac{q_1-1}{2}}}
\int_{\Delta_{q_1}}
\frac{d\vec{t}}
{\sqrt{G(A_n\1_{[t_1;t_2]}, \ldots, A_n\1_{[t_{q_1-1}; t_{q_1}]})}}\leq
$$
$$
q_1!
\frac{1}{(2\pi)^{\frac{q_1-1}{2}}}
\int_{\Delta_{q_1}}
\frac{\|A^{-1}_n\|^{q_1}}
{\sqrt{G(\1_{[t_1;t_2]}, \ldots, \1_{[t_{q_1-1}; t_{q_1}]})}}
d\vec{t}
\leq
$$
$$
\leq
q_1!\sup_{n\geq1}\|A^{-1}_n\|^{q_1}\cdot
$$
$$
\cdot
E\int_{\Delta_{q_1}}\prod^{q_1-1}_{i=1}\delta_0(w(t_{i})-w(t_{i+1}))d\vec{t}<+\infty.
$$
Using the same arguments one can conclude that
$$
E\int_{\mbR}l^x(u)^{q_2}du\leq
$$
$$
\leq\|A^{-1}_n\|^{q_2}
E\int_{\Delta_{q_2}}\prod^{q_2-1}_{i=1}\delta_0(w(t_{i+1})-w(t_{i}))d\vec{t}<+\infty.
$$
Therefore, the sequence $\{(l^{x_n})^k(l^x)^{2m-k}\}_{n\geq1}$ is uniformly integrable. It implies that in the case $0<k<2m$
$$
E\int_{\mbR}l^{x_n}(u)^kl^x(u)^{2m-k}du\to E\int_{\mbR}l^x(u)^{2m}du
$$
as $n\to\infty.$ Let $k=2m.$ Let us check  that for any $m\geq1$
$$
E\int_{\mbR}l^{x_n}(u)^{2m}du\to E\int_{\mbR}l^x(u)^{2m}du
$$
as $n\to\infty.$ Applying Theorem 3.3 one can obtain relation
$$
E\int_{\mbR}l^{x_n}(u)^{2m}du=
$$
$$
=
(2m)!
E\int_{\Delta_{2m}}\prod^{2m}_{i=1}\delta_0(x_n(t_{i+1})-x_n(t_{i}))d\vec{t}=
$$
$$
=
(2m)!\frac{1}{(2\pi)^{\frac{2m-1}{2}}}
\int_{\Delta_{2m}}
\frac{d\vec{t}}
{\sqrt{G(A_n\1_{[t_1;t_2]}, \ldots,A_n\1_{[t_{2m-1};t_{2m}]} )}}.
$$
Note that the integrand tends to
$$
\frac{1}
{\sqrt{G(A\1_{[t_1;t_2]}, \ldots,A\1_{[t_{2m-1};t_{2m}]} )}},
$$
as $n\to\infty.$
Since for any $n\geq1$
$$
\frac{1}
{\sqrt{G(A_n\1_{[t_1;t_2]}, \ldots,A_n\1_{[t_{2m-1};t_{2m}]} )}}\leq
$$
$$
\leq
\frac{\sup_{n\geq1}\|A^{-1}_n\|^{2m-1}}
{\sqrt{G(\1_{[t_1;t_2]}, \ldots,\1_{[t_{2m-1};t_{2m}]} )}}
$$
and
$$
\int_{\Delta_{2m}}\frac{d\vec{t}}
{(2\pi)^{\frac{2m-1}{2}}\sqrt{G(\1_{[t_1;t_2]}, \ldots,\1_{[t_{2m-1};t_{2m}]} )}}
$$
$$
=E\int_{\Delta_{2m}}
\prod^{2m-1}_{i=1}\delta_0(w(t_{i+1})-w(t_{i}))d\vec{t}<+\infty,
$$
then the Lebesgue dominated convergence theorem ends the proof.
\end{proof}
\section*{Appendix A. Properties of Gram determinant\\ $G(A1_{[0;t_1]}, \ldots, A1_{[0;t_k]})$}

In this appendix we collect properties of Gram determinant\break
$G(A1_{[0;t_1]}, \ldots, A1_{[0;t_k]})$ sufficient for the existence of moments of local time of integrator generated by the continuous linear operator $A$ in the space $L_2([0; 1])$ which satisfies conditions

1) $\dim\ker A<+\infty$

2) the restriction of operator $A$ on $(\ker A)^\perp$ is continuously invertible operator.

Let us start from a general Hilbert space $H$ and a continuous linear operator $A$ in $H.$

\noindent
{\bf Theorem A.1}. {\it
Suppose that $A$ is continuously invertible operator in the Hilbert space $H.$ Then for any elements $e_1,\ldots, e_n$ of the space $H$ the following relation holds
$$
G(Ae_1, \ldots, Ae_n)\geq\frac{1}{\|A^{-1}\|^{2n}}G(e_1, \ldots, e_n).
$$
}
\begin{proof}
Let $H_0$ be the linear span generated by elements $e_1, \ldots, e_n.$ Denote by $\wt{A}$ the restriction of operator $A$ on $H_0,$ i.e. $\wt{A}: H_0\to H.$ The polar decomposition \cite{9} for operator $\wt{A}$ is the following $\wt{A}=UT, $ where $T=\sqrt{\wt{A}^*\wt{A}}$ and
$$
U: H_0\to \wt{A}(H_0)=A(H_0)
$$
is isometric operator.

Consequently
$$
G(Ae_1, \ldots, Ae_n)=G(\wt{A}e_1,\ldots, \wt{A}e_n)=G(UTe_1, \ldots, UTe_n)=
$$
$$
=G(Te_1, \ldots, Te_n)=(\det T)^2G(e_1, \ldots, e_n).\eqno(A.1)
$$
Since $T$ is the self-adjoint operator, then there exists an orthonormal bases $g_1, \ldots, g_n$ in $H_0$ consisting eigenvectors of $T.$ Denote by $\lambda_1, \ldots, \lambda_n$ eigenvalues of $T.$ Then
$$
\det T=\prod^n_{i=1}\lambda_i. \eqno(A.2)
$$
The continuous invertibility of $T$ implies estimate
$$
\|Tg_i\|=|\lambda_i|\geq\frac{1}{\|T^{-1}\|}. \eqno(A.3)
$$
It follows from (A.1)--(A.3)  that
$$
G(Te_1, \ldots, Te_n)\geq\frac{1}{\|T^{-1}\|^{2n}}G(e_1, \ldots, e_n).\eqno(A.4)
$$
Since $\wt{A}=UT$ and
$$
\|\wt{A}z\|=\|Tz\|,\ z\in H_0,
$$
then
$$
\|T^{-1}\|=\sup_{\|Tz\|=1}\|z\|=\sup_{\|U^{-1}\wt{A}z\|=1}\|z\|=\sup_{\begin{subarray}{l}
\|Az\|=1,\\
z\in H_0
\end{subarray}
}
\|z\|. \eqno(A.5)
$$
On other hand
$$
\|A^{-1}\|=\sup_{\|Az\|=1}\|z\|.\eqno(A.6)
$$
It follows from (A.5) and (A.6) that
$$
\|T^{-1}\|\leq\|A^{-1}\|.\eqno(A.7)
$$
(A.4) and (A.7) imply that
$$
G(Ae_1, \ldots, Ae_n)\geq\frac{1}{\|A^{-1}\|^{2n}}G(e_1, \ldots, e_n).
$$
\end{proof}

\noindent
{\bf Theorem A.2}. {\it
Suppose that the continuous linear operator $A$ in $L_2([0; 1])$ satisfies conditions 1), 2) formulated at the beginning of Appendix A. Then the following integral
$$
\int_{\Delta_k}
\frac{d\vec{t}}
{\sqrt{G(A\1_{[0; t_1]}, \ldots, A\1_{[0; t_k]})}}
$$
converges.
}
\begin{proof}
Denote by $P$ a projection onto $\ker A.$ Then
$$
G(A\1_{[0; t_1]}, \ldots, A\1_{[0; t_k]})=
$$
$$
=G(A(I-P)\1_{[0; t_1]}, \ldots, A(I-P)\1_{[0; t_k]}).\eqno(A.8)
$$
Applying Theorem A.1 one can conclude that (A.8) greater or equal to
$$
\frac{1}{\|A^{-1}\|^{2k}}
G((I-P)\1_{[0; t_1]}, \ldots, (I-P)\1_{[0; t_k]}).\eqno(A.9)
$$
In proving the theorem we need the following statements which are of interest by itself.

{\bf Lemma A.3.} \cite{10} {\it
Let $\dim L<+\infty,\ e_1,\ldots, e_n$ be an orthogonal basis in $L.$  Then for any $g_1,\ldots, g_k\in L_2([0; 1])$ the following relation holds
$$
G((I-P_L)g_1, \ldots, (I-P_L)g_k)=G(g_1, \ldots, g_k, e_1, \ldots, e_m).
$$
}

\begin{proof}
Note tat
$$
G((I-P_L)g_1, \ldots, (I-P_L)g_k)=G((I-P_L)g_1, \ldots, (I-P_L)g_k, e_1, \ldots, e_m).
$$
Put
$$
c_{ij}=((I-P)g_i, (I-P)g_j)=(g_i, g_j)-\sum^m_{k=1}(e_k, g_i)(e_k, g_j).
$$
Then
$$
G((I-P)g_1, \ldots, (I-P)g_k, e_1, \ldots, e_m)=
\begin{vmatrix}
c_{11}&\ldots&c_{1k}&0&\ldots&0\\
c_{k1}&\ldots&c_{kk}&0&\ldots&0\\
0&\ldots&0&1&\ldots&0\\
\vdots&\ldots&\vdots&\vdots&&\vdots\\
0&\ldots&0&0&\ldots&1
\end{vmatrix}.
$$
On other hand
$$
G(g_1, \ldots, g_k, e_1, \ldots, e_m)=
$$
$$
=
\begin{vmatrix}
(g_1, g_1)&\ldots&(g_1, g_k)&(g_1, e_1)&\ldots&(g_1, e_m)\\
\vdots&\vdots&\vdots&\vdots&\vdots\\
(g_k, g_1)&\ldots&(g_k, g_k)&(g_k, e_1)&\ldots&(g_k, e_m)\\
(e_1, g_1)&\ldots&(e_1, g_k)&1&\ldots&0\\
\vdots&\vdots&\vdots&\vdots&\vdots\\
(e_m, g_1)&\ldots&(e_m, g_k)&0&\ldots&1
\end{vmatrix}.
\eqno(A.10)
$$
Multiplying $(k+1)$-th column by $(g_1, e_1), \ldots, (k+m)$-th column by $(g_1, e_m)$ and subtracting from 1-th column, $\ldots,$ $(k+1)$-th column by $(g_k, e_1), \ldots, (k+m)$-th column by $(g_k, e_m)$ and subtracting from $k$-th column we get that (A.10) equals
$$
\begin{vmatrix}
c_{11}&\ldots&c_{1k}&(g_1, e_1)&\ldots&(g_1,e_m)\\
\vdots&\vdots&\vdots&\vdots&\vdots&\vdots\\
c_{k1}&\ldots&c_{kk}&(g_k, e_1)&\ldots&(g_k,e_m)\\
0&\ldots&0&1&\ldots&0\\
\vdots&\vdots&\vdots&\vdots&\vdots&\vdots\\
0&\ldots&0&0&\ldots&1
\end{vmatrix}.\eqno(A.11)
$$
Multiplying $(k+1)$-th row by $(g_1, e_1), \ldots, (k+m)$-th row by $(g_1, e_n)$ and subtracting from 1-th row, $\ldots, (k+1)$-th row by $(g_k, e_k), \ldots, (k+m)$-th row by $(g_k, e_m)$ and subtracting from $k$-th row we get that (A.11) equals
$$
\begin{vmatrix}
c_{11}&\ldots&c_{1k}&0&\ldots&0\\
\vdots&\vdots&\vdots&\vdots&\vdots&\vdots\\
c_{k1}&\ldots&c_{kk}&0&\ldots&0\\
0&\ldots&0&1&\ldots&0\\
\vdots&\vdots&\vdots&\vdots&\vdots&\vdots\\
0&\ldots&0&0&\ldots&1
\end{vmatrix}
$$
which proves the lemma.
\end{proof}

Let $q_1,\ldots, q_p$ be an orthonormal basis in $\ker A.$ Then it follows from Lemma A.3 that the Gram determinant in (A.9) posesses representation
$$
G((I-P)\1_{[0; t_1]}, \ldots,(I-P)\1_{[0; t_k]} )=
G(\1_{[0; t_1]}, \ldots,\1_{[0; t_k]}, q_1, \ldots, q_p).
$$
Let us describe the set
$$
\{\vec{t}\in\Delta_k: \ G(\1_{[0; t_1]}, \ldots,\1_{[0; t_k]}, q_1, \ldots, q_p)=0\}.
$$
Note that
$$
G(\1_{[0; t_1]}, \ldots,\1_{[0; t_k]}, q_1, \ldots, q_p)=0
$$
iff there exists $\alpha_1, \ldots, \alpha_{k}$ such that $\alpha^2_1+\ldots+\alpha^2_{k}>0$ and $\beta_1, \ldots, \beta_p$ which satisfy relation
$$
\sum^k_{i=1}\alpha_i\1_{[0; t_i]}=\sum^p_{j=1}\beta_jq_j.
\eqno(A.12)
$$
Relation (A.12) implies that if $G(\1_{[0; t_1]}, \ldots,\1_{[0; t_k]}, q_1, \ldots, q_p)=0,$ then step functions belong to $\ker A.$ Let $L$ be the subspace of all step functions in $\ker A$ and $\{f_k,\ k=\ov{1,s}\}$ be an orthonormal basis in $L.$ Suppose that  $e_1, \ldots, e_m$ is an orthonormal basis in the orthogonal complement of $L$ in $\ker A.$ Note that $f_1, \ldots, f_s, e_1, \ldots, e_m$ is an orthonormal basis in $\ker A$ and for any $\beta_1, \ldots, \beta_m:\ \sum^m_{j=1}\beta_je_j\perp L.$ Let us check that
$$
\int_{\Delta_k}\frac{d\vec{t}}
{\sqrt{G(\1_{[0; t_1]}, \ldots,\1_{[0; t_k]}, f_1, \ldots, f_s, e_1, \ldots, e_m )}}<+\infty.
\eqno(A.13)
$$
To prove (A.13) we need the following statements.

\noindent
{\bf Lemma A.4.} \cite{10} {\it
Let $M$ be a set of step functions with the number of jumps less or equal to a fixed number. Then $M$ is a closed subset of $L_2([0; 1]).$
}
\begin{proof}
Suppose that $M$ is a set of step functions with the number of jumps less or equal to $n.$ Let $\{f_k,\ k\geq1\}\in M$ and $f_k\to f,\ k\to\infty.$  Check that $f\in M.$ Assume that the function $f_k$ has jumps at points
$
0<t^k_1<\ldots<t^k_{m_k}<1, \ 0\leq m_k\leq n.
$
If $m_k=0,$ then $f_k$ does not have jumps. By considering subsequence one can suppose that $m_k=m$ and
$$
(t^k_1, \ldots, t^k_m)\to (t_1, \ldots, t_m), \ k\to\infty,
$$
where $t_0=0\leq t_1\leq\ldots\leq t_m\leq 1=t_{m+1}.$ Denote by $\pi_{a,b}$ a projection onto $L_2([a; b]).$ If $t_i<t_{i+1}$ for some $i=\ov{0, m},$ then for any $\alpha, \beta$ such that $t_i<\alpha<\beta<t_{i+1}$ the following convergence holds $\pi_{\alpha, \beta}f_k\to \pi_{\alpha, \beta}f,\   k\to\infty.$
Consequently $f$ is a constant on any $[\alpha;\beta]\subset[t_i, t_{i+1}].$ It implies that $f$ is a constant on $[t_i; t_{i+1}].$ Using the same arguments for any $t_i<t_{i+1}$ we conclude that $f\in M.$
\end{proof}

\noindent
{\bf Lemma A.5.} \cite{10} {\it
There exists a positive constant $c$ such that the following relation holds
$$G(\1_{[0; t_1]}, \ldots,\1_{[0; t_k]}, f_1, \ldots, f_s, e_1, \ldots, e_m )\geq
c\cdot G(\1_{[0; t_1]}, \ldots,\1_{[0; t_k]}, f_1, \ldots, f_s).
$$
}
\begin{proof}
Note that
$$
\frac{G(\1_{[0; t_1]}, \ldots,\1_{[0; t_k]}, f_1, \ldots, f_s, e_1, \ldots, e_m )}
{G(\1_{[0; t_1]}, \ldots,\1_{[0; t_k]}, f_1, \ldots, f_s )}=
$$
$$
\frac{G(\1_{[0; t_1]}, \ldots,\1_{[0; t_k]}, f_1, \ldots, f_s, e_1)}
{G(\1_{[0; t_1]}, \ldots,\1_{[0; t_k]}, f_1, \ldots, f_s )}\cdot
$$
$$
\cdot
\frac{G(\1_{[0; t_1]}, \ldots,\1_{[0; t_k]}, f_1, \ldots, f_s, e_1, e_2 )}
{G(\1_{[0; t_1]}, \ldots,\1_{[0; t_k]}, f_1, \ldots, f_s, e_1 )}\cdot
$$
$$
\cdot\ldots\cdot
\frac{G(\1_{[0; t_1]}, \ldots,\1_{[0; t_k]}, f_1, \ldots, f_s, e_1, \ldots, e_m )}
{G(\1_{[0; t_1]}, \ldots,\1_{[0; t_k]}, f_1, \ldots, f_s, e_1, \ldots, e_{m-1} )}.
$$
Denote by
$$
\cK^i_{\vec{t}}=LS\{\1_{[0; t_1]}, \ldots,\1_{[0; t_k]}, f_1, \ldots, f_s, e_1, \ldots, e_i \},\ i=\ov{1,m}.
$$
Here by $LS\{g_1, \ldots, g_s\}$ we mean linear span generated by elements $g_1, \ldots, g_s.$

Put
$$
\cK^i=\cup_{\vec{t}\in\Delta_k}\cK^i_{\vec{t}}.
$$
Let $r_i$ be a distance from $e_i$ to $\cK^i,\ i=\ov{1, i-1}.$ Then to prove the lemma it  suffices to check that for any
$i=\ov{1, m} \ r_i>0.$ Suppose that it is not true. Then there exists $j=\ov{1, m}$ such that $r_j=0.$ Let $j=m.$ It implies that there exists the sequence
$$
\Big\{\sum^k_{i=1}\alpha^n_i\1_{[0; t^n_i]}+\sum^s_{j=1}\beta^n_jf_j+\sum^{m-1}_{l=1}\gamma^n_le_l,\ n\geq1\Big\}
$$
such that
$$
\Big\|e_m-\sum^{k-1}_{i=1}\alpha^n_i\1_{[0; t^n_i]}-\sum^s_{j=1}\beta^n_jf_j-\sum^{m-1}_{l=1}\gamma^n_le_l\Big\|\to0,\  n\to\infty.
$$

Therefore, the question is when
$$
\sum^k_{i=1}\alpha^n_i\1_{[0; t^n_i]}+\sum^s_{j=1}\beta^n_jf_j+\sum^{m-1}_{l=1}\gamma^n_le_l+e_m
$$
tends to zero as $n\to\infty?$ Consider possible cases:

1) Suppose that
$$
\varlimsup_{n\to\infty}\Big\|\sum^{m-1}_{l=1}\gamma^n_le_l+e_m\Big\|<+\infty.
$$
Considering a subsequence assume that for $l=\ov{1, m-1} \ \gamma^n_l\to\gamma_l,\ n\to\infty.$ It implies that
$$
\sum^k_{i=1}\alpha^n_i\1_{[0; t^n_i]}+\sum^s_{j=1}\beta^n_jf_j\to -\sum^{m-1}_{l=1}\gamma_le_l-e_m,\ n\to\infty.
$$
Note that $\sum^{m-1}_{l=1}\gamma_le_l+e_m$ is not a step function. On other hand
$$
\Big\{\sum^k_{i=1}\alpha^n_i\1_{[0; t^n_i]}+\sum^s_{j=1}\beta^n_jf_j,\  n\geq1\Big\}
$$
is a sequence of step functions with the number of jumps less or equal to a fixed number. It follows from Lemma A.4 that the case 1) is impossible.
Let us check that the case 2) is also impossible. Considering a subsequence one can suppose that
$$
a_n=\|\sum^{m-1}_{l=1}\gamma^n_le_l+e_m\|\to+\infty,\ n\to\infty
$$
and
$$
\frac{1}{a_n}\Big(\sum^{m-1}_{l=1}\gamma^n_le_l+e_m\Big)\to\sum^{m-1}_{l=1}p_le_l,\  n\to\infty,
$$
where $\|\sum^{m-1}_{l=1}p_le_l\|=1.$ Using the same arguments as in the case 1) one can get a contradiction.
\end{proof}
\noindent
{\bf Lemma A.6.} \cite{10} {\it
Let $0\leq s_1<\ldots<s_N\leq1$ be the points of jumps of functions $f_1, \ldots, f_s.$ Then there exists a positive constant $c_{\vec{s}}$ which depends on $\vec{s}=(s_1,\ldots, s_N)$ such that the following relation holds
$$
G(\1_{[0;t_1]}, \ldots, \1_{[0; t_k]}, f_1, \ldots, f_s)\geq
$$
$$
\geq
c_{\vec{s}}\cdot
G(\1_{[0;t_1]}, \ldots, \1_{[0; t_k]}, \1_{[0;s_1]}, \ldots, \1_{[0; s_N]}).
$$
}
\begin{proof}
Note that for any $i=\ov{1,s},\ f_i\in LS\{\1_{[0; s_i]},\ i=\ov{1, N}\}.$ Let us prove the statement of the lemma by induction. Let $\alpha_{ik}$ be a distance from $\1_{[0; t_i]}$ to $LS\{\1_{[0; t_{i+1}]}, \ldots, \1_{[0; t_k]}, f_1, \ldots, f_s\}$ and $\rho_{ik}$ be a distance from
$\1_{[0; t_i]}$ to $LS\{\1_{[0; t_{i+1}]}, \ldots, \1_{[0; t_k]}, \1_{[0;s_1]}, \ldots, \1_{[0; s_N]}\}.$ Then for  $k=1$
$$
G(\1_{[0; t_1]}, f_1, \ldots, f_s)=d^2_{11}G(f_1, \ldots, f_s)=d^2_{11}\geq\rho^2_{11}\geq
$$
$$
\geq\rho^2_{11}G
\Big(
\frac{\1_{[0; s_1]}}{\sqrt{s_1}},
\frac{\1_{[s_1; s_2]}}{\sqrt{s_2-s_1}}, \ldots,
\frac{\1_{[s_{N-1}; s_N]}}{\sqrt{s_N-s_{N-1}}}
\Big)=
$$
$$
=c_{\vec{s}}\cdot p^2_{11}
G(\1_{[0; s_1]}, \1_{[s_1,s_2]}, \ldots, \1_{[s_{N-1}; s_N]})=
$$
$$
=c_{\vec{s}}\cdot p^2_{11}
G(\1_{[0; s_1]}, \1_{[0,s_2]}, \ldots, \1_{[0; s_N]})=
$$
$$
=c_{\vec{s}}\cdot
G(\1_{[0; t_1]}, \1_{[0,s_1]}, \ldots, \1_{[0; s_N]}).
$$
Assume that the statement of lemma holds for an arbitrary $k.$ Let us prove that
$$
G(\1_{[0; t_1]}, \ldots,\1_{[0; t_{k+1}]}, f_1, \ldots, f_s)\geq
$$
$$
\geq
c_{\vec{s}}\cdot
G(\1_{[0; t_1]}, \ldots, \1_{[0; t_{k+1}]}, \1_{[0; s_1]}, \ldots, \1_{[0; s_N]}).
$$
Really,
$$
G(\1_{[0; t_1]}, \ldots,\1_{[0; t_{k+1}]}, f_1, \ldots, f_s)=
$$
$$
=d^2_{1k+1}
G(\1_{[0; t_2]}, \ldots,\1_{[0; t_{k+1}]}, f_1, \ldots, f_s)\geq
$$
$$
\geq
c_{\vec{s}}\cdot
\rho^2_{1k+1}
G(\1_{[0; t_2]}, \ldots,\1_{[0; t_{k+1}]}, \1_{[0; s_1]}, \ldots, \1_{[0; s_N]})
$$
$$
=c_{\vec{s}}\cdot
G(\1_{[0; t_1]}, \ldots,\1_{[0; t_{k+1}]}, \1_{[0; s_1]}, \ldots, \1_{[0; s_N]}).
$$
\end{proof}

Lemma A.3 -- Lemma A.6 were proved in our work \cite{10} but for the sake of clarity we recalled all the main steps of proofs. It follows from Lemma A.3 -- Lemma A.6 that to finish the proof of Theorem A.2 one must check that
$$
\int_{\Delta_k}
\frac{d\vec{t}}
{\sqrt{G(\1_{[0; t_1]}, \ldots,\1_{[0; t_{k+1}]}, \1_{[0; s_1]}, \ldots, \1_{[0; s_N]})}}<+\infty.\eqno(A.14)
$$

Consider the following partition of the domain $\Delta_k$
$$
\Delta_k=\cup_{k=n_0+\ldots+n_N}I_{n_0\ldots n_N},
$$
where
$$
I_{n_0\ldots n_N}=\{0\leq t_1\ldots\leq t_{n_0}\leq s_1\leq t_{n_0+1}\ldots\leq t_{n_0+n_1}\leq s_2\leq\ldots\leq s_n\leq
$$
$$
\leq t_{n_0+\ldots+n_{N-1}+1}\leq\ldots\leq t_k\leq1\}.
$$
Note that
$$
I_{n_0\ldots n_N}=\Delta_{n_0}(0; s_1)\times\Delta_{n_1}(s_1; s_2)\times\ldots\times\Delta_{n_N}(s_N; 1).\eqno(A.15)
$$
It follows from (A.15) that to finish the proof it suffices to check that
$$
\int_{\Delta_k(s_1,s_2)}
\frac{d\vec{t}}
{\sqrt{(t_1-s_1)(t_2-t_1)\cdot\ldots\cdot(t_k-t_{k-1})(s_2-t_k)}}<+\infty.
$$
The change of variables in the integral implies the following relation
$$
\int_{\Delta_k(s_1,s_2)}
\frac{d\vec{t}}
{\sqrt{(t_1-s_1)(t_2-t_1)\cdot\ldots\cdot(t_k-t_{k-1})(s_2-t_k)}}=
$$
$$
=c(\vec{s},k)\int_{\Delta_k}
\frac{d\vec{t}}
{\sqrt{t_1(t_2-t_1)\cdot\ldots\cdot(t_k-t_{k-1})(1-t_k)}}=
$$
$$
=c(\vec{s},k)
\int_{\Delta_k}
\frac{d\vec{t}}
{\sqrt{G(\1_{[0;t_1]},\1_{[0;t_2]}, \ldots, \1_{[0;t_k]}, \1_{[0;1]})}}=
$$
$$
=c(\vec{s},k)
\int_{\Delta_k}
\frac{d\vec{t}}
{\sqrt{G((I-\wt{P})\1_{[0;t_1]}, \ldots, (I-\wt{P})\1_{[0;t_k]})}}, \eqno(A.16)
$$
where $c(\vec{s},k)$ is the positive constant which depends on $s_1,s_2,k$
and $\wt{P}$ is a projection onto linear span generated by $\1_{[0;1]}.$

One can check that (A.16) equals
$$
\frac{c(\vec{s},k)}{k!(2\pi)^{k/2}}
E(l^{\wt{w}})^k.
$$
Here
$
\wt{w}(t)=w(t)-tw(1), \ t\in[0; 1]
$
is the Brownian bridge. To finish the proof it suffices to check that
$$
\sup_{k\geq1}E(l^{\wt{w}})^k<+\infty.
$$
Really, note that
$$
E(l^{\wt{w}})^k=EE((l^{w})^k/w(1)=0).
$$
It is known that joint probability density of random variables $l^w$ and $w(1)$ has the following representation \cite{11}
$$
p(a,b)=\frac{1}{\sqrt{2\pi}}(|b|+a)e^{-\frac{1}{2}(|b|+a)^2},\ a>0,\ b\in\mbR.
$$
Then
$$
E((l^w)^k/w(1)=0)=
\frac{\int^{+\infty}_{0}y^kp(y,0)dy}
{\int^{+\infty}_{0}p(y,0)dy}=
$$
$$
=
\frac{\int^{+\infty}_{0}y^{k+1}e^{-\frac{y^2}{2}}dy}
{\int^{+\infty}_{0}ye^{-\frac{y^2}{2}}dy}=\Gamma\Big(\frac{k}{2}+1\Big),
$$
where
$$
\Gamma(z)=\int^{+\infty}_{0}t^{z-1}e^{-t}dt,\ z\in\mbC,\ \Re z>0
$$
is the gamma function.
One can check that
$$
\int_{\Delta_k(s_1,s_2)}
\frac{d\vec{t}}
{\sqrt{(t_1-s_1)(t_2-t_1)\cdot\ldots\cdot(t_k-t_{k-1})(s_2-t_k)}}=
$$
$$
\frac{1}{k!}(2\pi)^{\frac{k}{2}}E(l^{\wt{w}_{s_1,s_2}})^k,
$$
were
$$
\wt{w}_{s_1,s_2}(t)=w(t-s_1)-\frac{t-s_1}{s_2-s_1}w(s_2-s_1),\ t\in[s_1;s_2]
$$
and $w(t),\ t\in[0;1]$ is a 1-dimensional Wiener process.
\end{proof}
\section*{Appendix B. On some relations between generalized functionals}

In this appendix we discuss conditions on elements $r_j\in L_2([0; 1]),\  j=\ov{1, n-1}$ that allow to establish relation
$$
\int_{\mbR}\prod^n_{k=1}\delta_0((f_k, \xi)-u)du=\prod^{n-1}_{j=1}\delta_0((r_j,\xi)),
$$
which is understood as equality of generalized functionals from white noise \cite{12} and will be checked using Fourier--Wiener transform \cite{13}.

\noindent
{\bf Theorem B.1.} \cite{10} {\it
Let $f_1, \ldots, f_n$ be the linearly independent elements in $L_2([0; 1]).$ Then
$$
\int_{\mbR}\prod^n_{k=1}\delta_0((f_k, \xi)-u)du=\prod^{n-1}_{k=1}(f_{k+1}-f_k, \xi).\eqno(B.1)
$$
}

\begin{proof}
To prove the statement let us calculate the Fourier--Wiener transform of the left-hand side and the right-hand side of equality (B.1). Denote by $\cT(\alpha)(h)$ the Fourier--Wiener transform of random variable $\alpha.$
One can check that
$$
\cT(\prod^{n-1}_{j=1}\delta_0((r_j,\xi)))(h)=
$$
$$
=
\frac{1}
{(2\pi)^{\frac{n-1}{2}\sqrt{G(r_1, \ldots, r_{n-1})}}}
e^{-\frac{1}{2}\|P_{r_1\ldots r_{n-1}}h\|^2} \eqno(B.2)
$$
(see \cite{13, 14}). Let us find the Fourier--Wiener transform of
$$
\int_{\mbR}\prod^n_{k=1}\delta_0((f_k, \xi)-u)du.
$$
$$
\cT(\int_{\mbR}\prod^{n}_{k=1}\delta_0((f_k,\xi)-u)du)(h)=
$$
$$
=\int_{\mbR}
\frac{1}
{(2\pi)^
{\frac{n}{2}\sqrt{G(f_1, \ldots, f_{n})}
}
}
e^{-\frac{1}{2}(B^{-1}(f_1,\ldots,f_n)u\vec{e}-\vec{a},u\vec{e}-\vec{a})}du, \eqno(B.3)
$$
where
$$
\vec{e}=\begin{pmatrix}
1\\
\vdots\\
1
\end{pmatrix},
\ \
\vec{a}=\begin{pmatrix}
(f_1,h)\\
\vdots\\
(f_n,h)
\end{pmatrix}.
$$

By integrating (B.3) over $u$ one can get
$$
\frac{1}
{(2\pi)^{\frac{n-1}{2}\sqrt{G(f_1, \ldots, f_{n})(B^{-1}(f_1, \ldots, f_{n})e, e)}}
}
$$
$$
\exp\Big\{-\frac{1}{2}(B^{-1}(f_1,\ldots,f_n)a,a)-
$$
$$
-\frac{(B^{-1}(f_1,\ldots,f_n)a,e)^2}
{(B^{-1}(f_1,\ldots,f_n)e,e)}
)\Big\}. \eqno(B.4)
$$
It is  not difficult to check that
$$
(B^{-1}(f_1, \ldots, f_n)a,a)=\|P_{f_1\ldots f_n}h\|^2.
$$
Consider function $f\in LS\{f_1, \ldots, f_n\}$ such that $(f, f_k)=1,\  k=\ov{1,n}.$ Then
$$
(B^{-1}(f_1,\ldots, f_n)\vec{e}, \vec{e})=\|P_{f_1\ldots f_n}f\|^2=\|f\|^2,
$$
$$
(B^{-1}(f_1,\ldots, f_n)\vec{a}, \vec{e})=(P_{f_1\ldots f_n}h,f).
$$
Therefore, (B.4) equals
$$
\frac{1}
{(2\pi)^{\frac{n-1}{2}}\sqrt{G(f_1,\ldots,f_n)}\|f\|}
e^{-\frac{1}{2}(\|P_{f_1\ldots f_n}h\|^2-\|P_fP_{f_1\ldots f_n}h\|^2)}.
$$
Denote by
$$
f^\perp=\{v\in LS\{f_1,\ldots,f_n\}: (v,f)=0\}.
$$
Then
$$
\cT\Big(\int_{\mbR}\prod^{n}_{k=1}\delta_0((f_k,\xi)-u)du\Big)(h)=
$$
$$
=
\frac{1}
{(2\pi)^
{\frac{n-1}{2}\sqrt{G(f_1, \ldots, f_{n})}\|f\|
}
}
e^{-\frac{1}{2}\|P_{f^\perp}h\|^2}. \eqno(B.5)
$$
By comparing (B.2)--(B.5) we obtain the following conditions on  elements $r_k,\  k=\ov{1, n-1}$

1) $LS\{r_1,\ldots, r_{n-1}\}=f^\perp$

2) $G(r_1,\ldots,r_{n-1})=G(f_1,\ldots,f_n)\|f\|^2.$

Let us check that $r_j=f_{j+1}-f_j$ satisfies conditions 1), 2). Really, put $M=LS\{f_2-f_1, \ldots, f_n-f_{n-1}\}.$ Then $f\perp M.$ Denote by $r$ the distance from $f_1$ to $M.$ One can see that
$$
G(f_1,\ldots,f_n)=G(f_1, f_2-f_1, \ldots, f_n-f_{n-1})=
$$
$$
=r^2G( f_2-f_1, \ldots, f_n-f_{n-1}).
$$
Since
$$
\Big(f_1, \frac{f}{\|f\|}\Big)=\|f_1\|\cos\alpha=r,
$$
then $r=\frac{1}{\|f\|}.$ Consequently,
$$
\|f\|^2G(f_1,\ldots,f_{n-1})=G( f_2-f_1, \ldots, f_n-f_{n-1}).
$$
\end{proof}
Lemma B.1 was proved in \cite{10} but for clarity we recalled the proof.
\renewcommand{\refname}{References}

\end{document}